\documentclass[a4paper]{amsart}
\usepackage{graphicx}
\usepackage{amssymb}
\usepackage{amsmath}
\usepackage{amsthm}
\usepackage{amscd}
\usepackage[all,2cell]{xy}

\usepackage[pagebackref,colorlinks]{hyperref}

\UseAllTwocells \SilentMatrices
\newtheorem{thm}{Theorem}[section]

\newtheorem{cor}[thm]{Corollary}

\newtheorem{lem}[thm]{Lemma}
\newtheorem{exm}[thm]{Example}

\newtheorem{prop}[thm]{Proposition}

\theoremstyle{definition}
\newtheorem{defn}[thm]{Definition}
\theoremstyle{remark}
\newtheorem{rem}[thm]{\bf Remark}
\numberwithin{equation}{section}

\begin{document}
\title[The Grothendieck group of a triangulated category]{The Grothendieck group of a triangulated category}
\author[Chen, Li, Zhang, Zhao] {Xiao-Wu Chen, Zhi-Wei Li$^*$, Xiaojin Zhang, Zhibing Zhao}

\makeatletter
\@namedef{subjclassname@2020}{\textup{2020} Mathematics Subject Classification}
\makeatother

\thanks{$^*$ The corresponding author}
\date{\today}
\subjclass[2020]{18G80, 18F30}

\thanks{xwchen$\symbol{64}$mail.ustc.edu.cn, zhiweili@jsnu.edu.cn, xjzhang@jsnu.edu.cn, zbzhao@ahu.edu.cn}
\keywords{Grothendieck group, triangulated category, silting subcategory, weight structure}

\maketitle

\dedicatory{}%
\commby{}%

\begin{abstract}
We give a direct proof of the following known result: the Grothendieck group of a triangulated category with a silting subcategory is isomorphic to the split Grothendieck group of the silting subcategory. Moreover, we obtain its cluster-tilting analogue.
\end{abstract}

\section{Introduction}

Let $\mathcal{T}$ be a skeletally small triangulated category. Denote by $\Sigma$ its suspension functor. Recall from \cite{AI} that a full additive subcategory $\mathcal{M}$ of $\mathcal{T}$ is \emph{presilting} if ${\rm Hom}_\mathcal{T}(\mathcal{M}, \Sigma^i\mathcal{M})=0$ for any $i\geq 1$, or equivalently, ${\rm Hom}_\mathcal{T}(M, \Sigma^i(M'))=0$ for any $M, M'\in \mathcal{M}$ and $i\geq 1$. It is called \emph{silting}, if in addition $\mathcal{T}={\rm tri}\langle \mathcal{M}\rangle$, that is, the smallest triangulated subcategory of $\mathcal{T}$ containing $\mathcal{M}$ coincides with $\mathcal{T}$ itself; compare \cite{KV}.

The definition given here is slightly different from \cite[Definition~2.1]{AI}, since we do not require that $\mathcal{M}$ is closed under direct summands. The main example in mind is the bounded  homotopy category $\mathbf{K}^b(\mathcal{A})$ of a skeletally small additive category $\mathcal{A}$.  It is clear that $\mathcal{A}$ is a silting subcategory of $\mathbf{K}^b(\mathcal{A})$, which  is  not necessarily closed under direct summands in general; see Lemma~\ref{lem:wis}.

The Grothendieck group of $\mathcal{T}$ is denoted by $K_0(\mathcal{T})$. For a skeletally small additive category $\mathcal{A}$,  we denote by $K_0^{\rm sp}(\mathcal{A})$ its split Grothendieck group.

The goal of this work is to give a direct proof of the following result; see Theorem~\ref{thm:main}.
\vskip 5pt

\noindent{\bf Theorem A.}\; \emph{Let $\mathcal{M}$ be a silting subcategory of $\mathcal{T}$. Then the inclusion $\mathcal{M}\hookrightarrow \mathcal{T}$ induces an isomorphism $K_0^{\rm sp}(\mathcal{M})\simeq K_0(\mathcal{T})$ of abelian groups.}

\vskip 5pt

Theorem~A is essentially due to \cite[Theorem~5.3.1]{Bond}, which is formulated using a weight structure and whose indirect proof relies on the weight complex functor. Under the additional Krull-Schmidt assumption on $\mathcal{T}$, Theorem~A is proved in \cite[Theorem~2.27]{AI}. We mention that \cite[Theorem~2.27]{AI} plays a  fundamental role in the study of K-theoretical aspects of  silting theory.

The surjectivity of the induced homomorphism $K_0^{\rm sp}(\mathcal{M})\rightarrow K_0(\mathcal{T})$ above is immediate, but the injectivity is somehow nontrivial. For this,  we establish the inverse homomorphism, whose argument modifies the one in \cite{AI} and  relies on the octahedral axiom (TR4).

Theorem~\ref{thm:main} contains slightly more information than Theorem~A, since the Grothendieck group of an intermediate subcategory is also studied. Moreover, we obtain a cluster-tilting  analogue of Theorem~A in Corollary~\ref{cor:cluster-tilting}, which describes the Grothendieck group of $\mathcal{T}$ as an explicit quotient group of the split Grothendieck group of a cluster-tilting subcategory. We  mention related work  \cite{Palu, Fed, JS, PPPP} on comparing the Grothendieck groups of triangulated categories and those of certain subcategories.

 Theorem~A has the following immediate  consequence \cite[Theorem~1.1]{Rose}, which seems to be well known to  experts  and is very related to \cite[Introduction, the fourth paragraph]{Sch} and \cite[Subsection~3.2.1, Lemma~3]{GS}.

\vskip 5pt
\noindent {\bf Corollary~B.}\; \emph{The inclusion $\mathcal{A}\hookrightarrow \mathbf{K}^b(\mathcal{A})$ induces an isomorphism $K_0^{\rm sp}(\mathcal{A})\simeq K_0(\mathbf{K}^b(\mathcal{A}))$ of abelian groups.}
\vskip 5pt

The paper is structured as follows. In Section~2, we study filtrations of objects with respect to a presilting subcategory. We prove Theorem~A in Section~3. In the final section, we study cluster-tilting analogues of the results in Section~3.

We refer to \cite{Hap, BBD} for triangulated categories and to \cite{Swan}  for Grothendieck groups. All subcategories are required to be full and additive.

\section{Filtrations}

Throughout this section, we fix  a triangulated category $\mathcal{T}$. We assume that  $\mathcal{M}\subseteq \mathcal{T}$ is a skeletally small additive subcategory, which is presilting.
 We  study filtrations on objects, which is  the key ingredient of the proof in the next section.

For two subcategories $\mathcal{X}$ and $\mathcal{Y}$, we have the following subcategory
$$\mathcal{X}*\mathcal{Y}=\{E\in \mathcal{T}\;|\;\exists  \mbox{ an exact triangle } X\rightarrow E\rightarrow Y\rightarrow \Sigma(X) \mbox{ with } X\in \mathcal{X}, Y\in \mathcal{Y}\}.$$
The operation $*$ on subcategories is associative; see \cite[Lemme~1.3.10]{BBD}.

\begin{lem}\label{lem:1}
		The following statements hold.
		\begin{enumerate}
			\item $\Sigma^{i}\mathcal{M}*\Sigma^j\mathcal{M}\subseteq \Sigma^{j}\mathcal{M}*\Sigma^i\mathcal{M}$ for $j< i$, and $\Sigma^{i}\mathcal{M}*\Sigma^i\mathcal{M}=\Sigma^i\mathcal{M}$.
			\item ${\rm Hom}_\mathcal{T}(\Sigma^{-n}\mathcal{M}*\Sigma^{-(n-1)}\mathcal{M}*\cdots *\Sigma^{-1}\mathcal{M}, \Sigma^m\mathcal{M})=0$  if $0\leq m$ and $1\leq n$.
		\end{enumerate}
		
		\begin{proof}
			For (1), we consider an exact triangle
			$$\Sigma^{i}(M_1)\longrightarrow E\longrightarrow \Sigma^{j}(M_2)\stackrel{a}\longrightarrow \Sigma^{i+1}(M_1)$$
			with $M_i\in \mathcal{M}$. Since $\mathcal{M}$ is presilting and $j\leq i$, we have $a=0$. It follows that $E\simeq \Sigma^i(M_1)\oplus \Sigma^j(M_2)$, which belongs to $\Sigma^{j}\mathcal{M}*\Sigma^i\mathcal{M}$. If $i=j$, the object $E$ belongs to $\Sigma^i\mathcal{M}$.
			
			For (2), we take $0\leq m$,  and consider the subcategory
			$$\mathcal{S}_m=\{E\in \mathcal{T}\; |\; {\rm Hom}_\mathcal{T}(E, \Sigma^m\mathcal{M})=0 \}.$$
			This subcategory is closed under extensions. Since $\mathcal{M}$ is presilting, $\mathcal{S}_m$ contains $\Sigma^{-n}\mathcal{M}$  for any $1\leq n$. Then we deduce (2).
		\end{proof}
\end{lem}

\begin{defn}
Let $X$ be an object in $\mathcal{T}$. A \emph{$\Sigma^{\leq 0}(\mathcal{M})$-filtration} of \emph{length} $n\geq 1$ for $X$ means a sequence of morphisms
$$0=X_n\longrightarrow X_{n-1}\longrightarrow \cdots \longrightarrow X_1\longrightarrow X_0=X$$
such that each morphism fits into an exact triangle
$$X_{i+1}\longrightarrow X_i\longrightarrow \Sigma^{-i}(M^X_i)\longrightarrow \Sigma(X_{i+1})$$
with the $i$-th \emph{factors} $M_i^X\in \mathcal{M}$ for each $0\leq i\leq n-1$.
\end{defn}

We denote by $\mathcal{F}$ the full subcategory of $\mathcal{T}$ formed by those objects, which  admit a $\Sigma^{\leq 0}(\mathcal{M})$-filtration.

\begin{rem}\label{rem:1}
(1) In the filtration above,  each $X_i$ belongs to
$$\Sigma^{-(n-1)}\mathcal{M}*\cdots *\Sigma^{-(i+1)}\mathcal{M}*\Sigma^{-i}\mathcal{M}.$$
Consequently,   by Lemma~\ref{lem:1}(2) we have
$${\rm Hom}_\mathcal{T}(X, \Sigma(M))=0={\rm Hom}_\mathcal{T}(X_1, M)$$
for any  $M\in \mathcal{M}$.

(2) We observe that
  $$\mathcal{F}=\bigcup_{n\geq 0} \Sigma^{-n}\mathcal{M}*\cdots *\Sigma^{-1}\mathcal{M}*\mathcal{M}.$$
  By applying Lemma~\ref{lem:1}(1) repeatedly, we infer that $\mathcal{F}$ is closed under extensions.
\end{rem}

Let $\mathcal{A}$ be a skeletally small additive category. For each object $A$, the corresponding element in the split Grothendieck group $K_0^{\rm sp}(\mathcal{A})$ is denoted by $\langle A\rangle$. Therefore, we have $\langle A\oplus B\rangle =\langle A\rangle+\langle B\rangle$.

Assume that there are two  $\Sigma^{\leq 0}(\mathcal{M})$-filtrations of $X$:
\begin{align}\label{filt1}
0=X_n\longrightarrow X_{n-1}\longrightarrow \cdots \longrightarrow X_1\longrightarrow X_0=X
\end{align}
and
\begin{align}\label{filt2}
0=Y_m\longrightarrow Y_{m-1}\longrightarrow \cdots \longrightarrow Y_1\longrightarrow Y_0=X
\end{align}
with factors $M_i^X$ and $M_j^Y$. The two filtrations are said to be \emph{equivalent} if
$$\sum_{i=0}^{n-1} (-1)^i \langle M^X_i\rangle=\sum_{j=0}^{m-1} (-1)^j \langle M^Y_j\rangle$$
holds in $K_0^{\rm sp}(\mathcal{M})$.

The argument in the following proof resembles the one in proving the Jordan-H{\"o}lder theorem for modules of finite length. It releases the restriction of the existence of minimal morphisms, which is needed in the proof of \cite[Theorem~2.27]{AI}; compare \cite[Remark~5.3]{Fed}.

\begin{prop}\label{prop:filt}
    Any two $\Sigma^{\leq 0}(\mathcal{M})$-filtrations of an object $X$  are equivalent.
\end{prop}

\begin{proof}
We assume that (\ref{filt1}) and (\ref{filt2}) are two given filtrations of $X$ with $n, m\geq 1$. By extending one of the filtrations by zeros, we may assume that they have the same length, that is, $n=m$.  We use induction on the common length $n$. If $n=1$, the statement is trivial, since both $M_0^X$ and $M_0^Y$ are isomorphic to $X$.

We assume that $n\geq 2$. We apply (TR4) to the exact triangles $Y_1\rightarrow X\rightarrow M_0^Y\rightarrow \Sigma(Y_1)$ and $X\rightarrow M_0^X\rightarrow \Sigma(X_1)\rightarrow \Sigma(X)$, and obtain the following commutative diagram.
\[\xymatrix{
Y_1\ar@{=}[d]\ar[r] & X\ar[d] \ar[r] & M_0^Y\ar[d] \ar[r] & \Sigma(Y_1)\ar@{=}[d]\\
Y_1\ar[r]^a & M_0^X \ar[d]\ar[r] & Z\ar[d] \ar[r] & \Sigma(Y_1)\ar[d]\\
& \Sigma(X_1)\ar[d] \ar@{=}[r] & \Sigma(X_1)\ar[d]^-{\Sigma(b)} \ar[r] & \Sigma(X)\\
& \Sigma(X) \ar[r] & \Sigma(M_0^Y)
}\]
By Remark~\ref{rem:1}, we have $a=0=b$. Therefore, we have isomorphisms
$$\Sigma(X_1)\oplus M_0^Y\simeq Z\simeq \Sigma(Y_1)\oplus M_0^X.$$

The exact triangle $X_2\rightarrow X_1\rightarrow \Sigma^{-1}(M_1^X)\rightarrow \Sigma(X_2)$ gives rise to the following one
$$\Sigma(X_2)\longrightarrow Z\longrightarrow M_1^X\oplus M_0^Y\longrightarrow \Sigma^2(X_2).$$
Consequently, we have a $\Sigma^{\leq 0}(\mathcal{M})$-filtration of length $n-1$ for $Z$.
$$0=\Sigma(X_n)\longrightarrow \Sigma(X_{n-1})\longrightarrow \cdots \longrightarrow \Sigma(X_2)\longrightarrow Z$$
Its factors are given by $\{M_1^X\oplus M_0^Y, M_2^X, \cdots, M_{n-1}^X\}$. Similarly, we have another filtration of length $n-1$
$$0=\Sigma(Y_n)\longrightarrow \Sigma(Y_{n-1})\longrightarrow \cdots \longrightarrow \Sigma(Y_2)\longrightarrow Z$$
with factors $\{M_1^Y\oplus M_0^X, M_2^Y, \cdots, M_{n-1}^Y\}$. Now by induction, these two filtrations for $Z$ are equivalent, that is, we have
$$\langle M_1^X\oplus M_0^Y\rangle +\sum_{i=2}^{n-1} (-1)^{i-1}\langle M_i^X\rangle=\langle M_1^Y\oplus M_0^X\rangle +\sum_{j=2}^{n-1} (-1)^{j-1}\langle M_j^Y\rangle.$$
This implies that $\sum_{i=0}^{n-1} (-1)^i \langle M^X_i\rangle=\sum_{j=0}^{m-1} (-1)^j \langle M^Y_j\rangle$, as required.
\end{proof}

The following result is analogous to the horseshoe lemma.

\begin{lem}\label{lem:HL}
     Let  $X\stackrel{a}\rightarrow Y \stackrel{b}\rightarrow Z\stackrel{c}\rightarrow \Sigma(X)$ be an exact triangle with $X, Z\in \mathcal{F}$, and assume that $n\geq 1$.  If
     $$0=X_n\longrightarrow X_{n-1}\longrightarrow \cdots \longrightarrow X_1\longrightarrow X_0=X
$$
and
$$0=Z_n\longrightarrow Z_{n-1}\longrightarrow \cdots \longrightarrow Z_1\longrightarrow Z_0=Z
$$
are $\Sigma^{\leq 0}(\mathcal{M})$-filtrations of $X$ and $Z$, respectively, then $Y$ has a $\Sigma^{\leq 0}(\mathcal{M})$-filtration
$$0=Y_n\longrightarrow Y_{n-1}\longrightarrow \cdots \longrightarrow Y_1\longrightarrow Y_0=Y
$$
with its factors $M_i^Y\simeq M_i^X\oplus M_i^Z$ for $0\leq i\leq n-1$.
\end{lem}

\begin{proof}
By Remark~\ref{rem:1}, the following square trivially commutes.
\[\xymatrix{
Z\ar[d] \ar[r]^-c & \Sigma(X)\ar[d]\\
M_0^Z\ar[r]^-0 & \Sigma(M_0^X)
}\]
Applying the $3\times 3$ Lemma in \cite[Proposition~1.1.11]{BBD} and rotations, we have the following commutative diagram with exact columns and rows.
\[\xymatrix{
\Sigma^{-1}(M_0^X)\ar[d] \ar[r] & \Sigma^{-1}(M_0^X\oplus M_0^Z) \ar[d]\ar[r] & \Sigma^{-1}(M_0^Z) \ar[d] \ar[r] & M_0^X\ar[d]\\
X_1 \ar[d]\ar[r]^{a_1} & Y_1 \ar[d] \ar[r]^{b_1} & Z_1 \ar[d]\ar[r]^{c_1} & \ar[d] \Sigma(X_1)\\
X \ar[d] \ar[r]^{a} & Y \ar[d] \ar[r]^{b} & Z\ar[r]^{c} \ar[d] & \Sigma(X) \ar[d]\\
M_0^X \ar[r] & M_0^X\oplus M_0^Z \ar[r] & M_0^Z \ar[r]^0 & \Sigma(M_0^X)
}\]
The middle vertical triangle
$$\Sigma^{-1}(M_0^X\oplus M_Z^0)\longrightarrow Y_1\longrightarrow Y\longrightarrow M_0^X\oplus M_0^Z$$ implies that $M_0^Y\simeq M_0^X\oplus M_0^Z$. We now repeat the argument to the exact triangle $X_1\stackrel{a_1}\rightarrow Y_1 \stackrel{b_1}\rightarrow Z_1\stackrel{c_1}\rightarrow \Sigma(X_1)$. Then we obtain the required filtration for $Y$.
\end{proof}

\section{The proof of Theorem~A}
In this section, we give the proof of Theorem~A and describe the original version \cite{Bond} of Theorem~A in terms of bounded weight structures. We fix a skeletally small triangulated category $\mathcal{T}$.

Let $\mathcal{C}$ be a full additive subcategory of $\mathcal{T}$. We define its \emph{Grothendieck group} $K_0(\mathcal{C})$ to be the abelian group generated by $\{[C]\; | \; C\in \mathcal{C}\}$ subject to the relations $[C]-([C_1]+[C_2])$ whenever there is an exact triangle $C_1\rightarrow C\rightarrow C_2\rightarrow \Sigma(C_1)$ in $\mathcal{T}$ with $C_i, C\in \mathcal{C}$. We emphasize that $K_0(\mathcal{C})$ depends on the inclusion $\mathcal{C}\hookrightarrow \mathcal{T}$.

The following result indicates that the Grothendieck group $K_0(\mathcal{C})$ of a certain subcategory $\mathcal{C}$ might be  useful.

\begin{lem}\label{lem:Sigma}
    Assume that the full subcategory $\mathcal{C}$ is closed under $\Sigma^{-1}$ and that for any object $X\in \mathcal{T}$ there exists a natural number $n$ such that $\Sigma^{-n}(X)\in \mathcal{C}$. Then the inclusion $\mathcal{C}\hookrightarrow \mathcal{T}$ induces an isomorphism $K_0(\mathcal{C})\simeq K_0(\mathcal{T})$.
\end{lem}

\begin{proof}
    We make an easy observation: for each object $C$ in $\mathcal{C}$, the trivial triangle $\Sigma^{-1}(C)\rightarrow 0\rightarrow C\rightarrow C$ implies that $[\Sigma^{-1}(C)]=-[C]$ in $K_0(\mathcal{C})$. For each object $X$ in $\mathcal{T}$, we choose a natural number $n$ with $\Sigma^{-n}(X)\in \mathcal{C}$, and  define an element $\phi(X)=[\Sigma^{-n}(X)]$ in $K_0(\mathcal{C})$. The observation above implies that $\phi(X)$ does not depend on the choice of $n$. Since any $\Sigma^{-n}$ is a triangle functor, these $\phi(X)$ give rise to a well-defined homomorphism $\Phi\colon K_0(\mathcal{T})\rightarrow K_0(\mathcal{C})$ such that $\Phi([X])=\phi(X)$. It is routine to verify that $\Phi$ is inverse to the induced homomorphism $K_0(\mathcal{C})\rightarrow  K_0(\mathcal{T})$.
\end{proof}

\begin{prop}\label{prop:F}
 Let $\mathcal{M}$ be a presilting subcategory of $\mathcal{T}$. Then the inclusion $\mathcal{M}\hookrightarrow \mathcal{F}$ induces an isomorphism $K_0^{\rm sp}(\mathcal{M})\simeq K_0(\mathcal{F})$ of abelian groups.
\end{prop}

\begin{proof}
For each $X\in \mathcal{F}$, we choose a $\Sigma^{\leq 0}(\mathcal{M})$-filtration
$$0=X_n\longrightarrow X_{n-1}\longrightarrow \cdots \longrightarrow X_1\longrightarrow X_0=X$$
with $n\geq 1$ and factors $M_i^X$.
We define an element
$$\gamma(X)=\sum_{i=0}^{n-1}(-1)^i\langle M_i^X\rangle$$
in $K_0^{\rm sp}(\mathcal{M})$. By Proposition~\ref{prop:filt}, the element $\gamma(X)$ does not depend on the choice of such a filtration. By Lemma~\ref{lem:HL}, the map $(X\mapsto \gamma(X))$ is compatible with exact triangles in $\mathcal{F}$. Therefore, such a map induces a well-defined homomorphism $\Gamma\colon K_0(\mathcal{F})\rightarrow K_0^{\rm sp}(\mathcal{M})$ such that $\Gamma([X])=\gamma(X)$. It is routine to verify that $\Gamma$ is inverse to the induced homomorphism $K_0^{\rm sp}(\mathcal{M})\rightarrow K_0(\mathcal{F})$.
\end{proof}

The following result contains Theorem~A, which is analogous to \cite[Proposition~4.11]{PPPP} in the setting of extriangulated categories \cite{NP}.

\begin{thm}\label{thm:main}
	 Let $\mathcal{M}$ be a silting subcategory of $\mathcal{T}$. Then the inclusions $\mathcal{M}\hookrightarrow \mathcal{F}\hookrightarrow \mathcal{T}$ induce isomorphisms $K_0^{\rm sp}(\mathcal{M})\simeq K_0(\mathcal{F})\simeq K_0(\mathcal{T})$ of abelian groups.
\end{thm}

\begin{proof}
    The isomorphism $K_0^{\rm sp}(\mathcal{M})\rightarrow K_0(\mathcal{F})$ follows from Proposition~\ref{prop:F}. Recall from  Remark~\ref{rem:1}(2) that
     $$\mathcal{F}=\bigcup_{n\geq 0} \Sigma^{-n}\mathcal{M}*\cdots *\Sigma^{-1}\mathcal{M}*\mathcal{M}.$$
In particular, $\mathcal{F}$ is closed under $\Sigma^{-1}$.     Since $\mathcal{T}={\rm tri}\langle \mathcal{M}\rangle$, each object $X$ of $\mathcal{T}$ belongs to
     $$\Sigma^{i_1}\mathcal{M}* \cdots *\Sigma^{i_{n-1}}\mathcal{M}* \Sigma^{i_n}\mathcal{M}$$
     for some $i_1, \cdots, i_{n-1}, i_n\in \mathbb{Z}$.
     By Lemma~\ref{lem:1}(1), we may assume that $i_1< i_2<\cdots < i_n $. Consequently, for any sufficiently large $n$, the object $\Sigma^{-n}(X)$
    belongs to $\mathcal{F}$. So, the conditions in Lemma~\ref{lem:Sigma} are fulfilled.  Then the required isomorphism $K_0(\mathcal{F})\simeq K_0(\mathcal{T})$ follows immediately.
\end{proof}

Recall from \cite[Definition~1.1.1]{Bond} that a \emph{weight structure} on $\mathcal{T}$ is a pair $(\mathcal{U}_{\geq 0}, \mathcal{U}_{\leq 0})$ of subcategories subject to the following conditions:
\begin{enumerate}
    \item Both $\mathcal{U}_{\geq 0}$ and $\mathcal{U}_{\leq 0}$ are closed under direct summands;
    \item $\mathcal{U}_{\geq 0}$ is closed under $\Sigma^{-1}$, and $\mathcal{U}_{\leq 0}$ is closed under $\Sigma$;
    \item ${\rm Hom}_\mathcal{T}(\mathcal{U}_{\geq 0}, \Sigma \mathcal{U}_{\leq 0})=0$;
    \item $\mathcal{U}_{\geq 0}* \Sigma \mathcal{U}_{\leq 0}=\mathcal{T}$.
\end{enumerate}
The \emph{core} of the weight structure is defined to be the subcategory $\mathcal{C}=\mathcal{U}_{\geq 0}\cap \mathcal{U}_{\leq 0}$. It is a presilting subcategory of $\mathcal{T}$. We mention that a weight structure is called a co-t-structure, and the core is called the coheart in \cite[Definitions~2.4 and 2.6]{Pau}.

The weight structure $(\mathcal{U}_{\geq 0}, \mathcal{U}_{\leq 0})$ is \emph{bounded} if for each object $X$, there exist natural numbers $n\leq m$ such that $X\in \Sigma^{-n}\mathcal{U}_{\geq 0}\cap \Sigma^{-m}\mathcal{U}_{\leq 0}$. In this case, the core $\mathcal{C}$ is a silting subcategory; see \cite[Corollary~1.5.7]{Bond}. Moreover, by \cite[Proposition~2.23(b)]{AI} any silting subcategory which is closed under direct summands arises as the core of a bounded weight structure.

The following result   is  due to \cite[Theorem~5.3.1]{Bond}, which might be viewed as a version of  Theorem~\ref{thm:main}.

\begin{cor}
Let $(\mathcal{U}_{\geq 0}, \mathcal{U}_{\leq 0})$ be a bounded weight structure on $\mathcal{T}$ with core $\mathcal{C}$. Then the inclusions $\mathcal{C}\hookrightarrow \mathcal{U}_{\geq 0}\hookrightarrow \mathcal{T}$ induce isomorphisms $K_0^{\rm sp}(\mathcal{C})\simeq K_0(\mathcal{U}_{\geq 0})\simeq K_0(\mathcal{T})$ of abelian groups.
\end{cor}

\begin{proof}
    As mentioned above, the core $\mathcal{C}$ is a silting subcategory of $\mathcal{T}$. Moreover, by \cite[Proposition~2.23(b)]{AI} an object has a $\Sigma^{\leq 0}(\mathcal{C})$-filtration if and only if it belongs to $\mathcal{U}_{\geq 0}$. Then we deduce these isomorphisms by Theorem~\ref{thm:main}.
\end{proof}

\begin{rem}
(1) By applying the  corollary above to the opposite category of $\mathcal{T}$, one might deduce isomorphisms $K_0^{\rm sp}(\mathcal{C})\simeq K_0(\mathcal{U}_{\leq 0})\simeq K_0(\mathcal{T})$ of abelian groups.

 (2) The corollary above is analogous to the following well-known result; see \cite[Proposition~A.9.5]{Achar}. Let $\mathcal{T}$ have  a bounded t-structure $(\mathcal{D}^{\leq 0}, \mathcal{D}^{\geq 0})$ with heart $\mathcal{A}=\mathcal{D}^{\leq 0}\cap \mathcal{D}^{\geq 0}$. Then the inclusions $\mathcal{A}\hookrightarrow \mathcal{D}^{\leq 0}\hookrightarrow \mathcal{T}$ induce isomorphisms $K_0(\mathcal{A})\simeq K_0(\mathcal{D}^{\leq 0})\simeq K_0(\mathcal{T})$ of abelian groups.

 (3) We mention that the isomorphism $K_0^{\rm sp}(\mathcal{C})\simeq K_0(\mathcal{T})$ above is extended to isomorphisms between the corresponding higher $K$-groups in \cite{So}. One expects that the higher $K$-groups of $\mathcal{U}_{\geq 0}$ are also isomorphic to them.
\end{rem}

Let $\mathcal{A}$ be a skeletally small additive category. Denote its bounded  homotopy category by $\mathbf{K}^b(\mathcal{A})$. We identify any object in $\mathcal{A}$ with the corresponding stalk complex concentrated in degree zero. Therefore, $\mathcal{A}$ is viewed as a full subcategory of $\mathbf{K}^b(\mathcal{A})$.  Moreover, it is  a silting subcategory.

Recall that an idempotent $e\colon A\rightarrow A$ in $\mathcal{A}$ splits if there are morphisms $r\colon A\rightarrow Y$ and $s\colon Y\rightarrow A$ satisfying $e=s\circ r$ and ${\rm Id}_Y=r\circ s$. The  category $\mathcal{A}$ is said to be \emph{weakly idempotent-split},  if any idempotent $e\colon X\rightarrow X$ splits whenever ${\rm Id}_X-e$ splits.

The following result is due to \cite[Theorem~4.1 and Corollary~4.3(1)]{BV}.

\begin{lem}\label{lem:wis}
The subcategory $\mathcal{A}\subseteq \mathbf{K}^b(\mathcal{A})$ is closed under direct summands if and only if $\mathcal{A}$ is weakly idempotent-split.
\end{lem}

\begin{proof}
By \cite[Lemma~2.2]{LC}, any triangulated category is weakly idempotent-split. Consequently, any full subcategory of a triangulated category is weakly idempotent-split, provided that it is closed under direct summands. Then we have the ``only if" part.

For the ``if" part,  we observe by  \cite[Theorem~4.1]{BV} that $\mathcal{A}$ is identified with the core of the standard weight structure on $\mathbf{K}^b(\mathcal{A})$. In particular, it is closed under direct summands in $\mathbf{K}^b(\mathcal{A})$.
\end{proof}

\section{A  cluster-tilting analogue of Theorem~A}

In this section, we obtain a cluster-tilting analogue of Theorem~A; see Corollary~\ref{cor:cluster-tilting}. The main result is Theorem~\ref{thm:Fd}, which is  a cluster-tilting analogue of Proposition~\ref{prop:F}.

Throughout this section, we fix   $d\in \{2, 3, \cdots\}$, and $\mathcal{T}$ a skeletally small triangulated category. Following \cite[Section~3]{IYo}, we say that a full additive subcategory  $\mathcal{M}$ of $\mathcal{T}$ is \emph{$d$-rigid} if ${\rm Hom}_\mathcal{T}(\mathcal{M},\Sigma^i\mathcal{M})=0$ for each $1\leq i<d$. We mention that a presilting subcategory is always $d$-rigid.

We fix a $d$-rigid subcategory $\mathcal{M}$. The following result is analogous to Lemma~\ref{lem:1} with the same proof.

\begin{lem}\label{lem:ct1}
		The following statements hold.
		\begin{enumerate}
			\item $\Sigma^{i}\mathcal{M}*\Sigma^j\mathcal{M}\subseteq \Sigma^{j}\mathcal{M}*\Sigma^i\mathcal{M}$ for $i+1-d<  j< i$, and $\Sigma^{i}\mathcal{M}*\Sigma^i\mathcal{M}=\Sigma^i\mathcal{M}$.
			\item ${\rm Hom}_\mathcal{T}(\Sigma^{-n}\mathcal{M}*\Sigma^{-(n-1)}\mathcal{M}*\cdots *\Sigma^{-1}\mathcal{M}, \Sigma^m\mathcal{M})=0$  if $0\leq m<d-1$ and $1\leq n< d-m$. \hfill $\square$
		\end{enumerate}
\end{lem}

For each $1\leq m\leq d$, we  consider the full subcategory $\mathcal{F}_m$ of $\mathcal{T}$ formed by those objects, which  admit a $\Sigma^{\leq 0}(\mathcal{M})$-filtration of length $n$ with $n\leq m$. Therefore, we have
$$\mathcal{M}=\mathcal{F}_1\subseteq \mathcal{F}_2\subseteq \cdots \subseteq \mathcal{F}_d.$$
By Lemma~\ref{lem:ct1}(1), for each $1\leq m<d$, the subcategory $\mathcal{F}_m$ is   closed under extensions; compare Remark~\ref{rem:1}(2). However, $\mathcal{F}_d$ is not closed under extensions in general; compare Lemma~\ref{lem:fd-ext} below.

For any object $X$ in $\mathcal{T}$, we denote by ${\rm add}\; X$ the full subcategory formed by direct summands of finite direct sums of $X$.

\begin{exm}
 {\rm Let $\mathbb{K}$ be a field. Let $A$ be the finite dimensional $\mathbb{K}$-algebra given by the following quiver
 \[\xymatrix{1 \ar[r]^\alpha & 2 \ar[r]^-\beta &3}\]
 subject to the relation $\beta\alpha=0$. Each vertex $i$ corresponds to a simple module $S_i$ and an indecomposable projective module $P_i$. Denote by $\mathbf{D}^b(A\mbox{-mod})$ the bounded derived category of finite dimensional left $A$-modules. Set $\mathcal{M}={\rm add}\; (S_1\oplus S_3)$. Then it is a $2$-rigid subcategory of $\mathbf{D}^b(A\mbox{-mod})$. We have
 $$\mathcal{F}_2=\Sigma^{-1}\mathcal{M}\ast \mathcal{M}={\rm add}\; (\Sigma^{-1}(S_1\oplus S_3)\oplus (S_1\oplus S_3)).$$
 Consider the two-term complex
 $$X=\cdots \longrightarrow 0\longrightarrow P_2\longrightarrow P_1\longrightarrow 0\longrightarrow \cdots,$$
 where $P_1$ has degree $1$ and the unnamed arrow $P_2\rightarrow P_1$ is induced by multiplying $\alpha$ from the right. We have an exact triangle
 $$S_3\longrightarrow X\longrightarrow \Sigma^{-1}(S_1)\longrightarrow \Sigma(S_3).$$
Therefore,  $X$ belongs to $\mathcal{M}\ast \Sigma^{-1}\mathcal{M}\subseteq \mathcal{F}_2\ast \mathcal{F}_2$, but $X$ does not belong to $\mathcal{F}_2$. Consequently, $\mathcal{F}_2$ is not closed under extensions. }
\end{exm}

The following fact is well known.

\begin{lem}\label{lem:fd-ext}
Let $\mathcal{M}$ be a $d$-rigid subcategory. Then $\mathcal{F}_d$ is closed under extensions if and only if $\mathcal{M}\ast \Sigma^{-(d-1)}\mathcal{M}\subseteq \mathcal{F}_d$.
\end{lem}

\begin{proof}
It suffices to prove the ``if" part. Since $\mathcal{F}_d= \Sigma^{-(d-1)}\mathcal{M}\ast \cdots \ast \Sigma^{-1}\mathcal{M}\ast\mathcal{M}$, by applying Lemma~\ref{lem:ct1}(1) repeatedly, we have
\begin{align}\label{equ:ct1}
\Sigma^{-i}\mathcal{M}\ast \mathcal{F}_d\subseteq \mathcal{F}_d \mbox{ and } \mathcal{F}_d\ast \Sigma^{-j}\mathcal{M}\subseteq \mathcal{F}_d
\end{align}
for any $1\leq i\leq d-1$ and $0\leq j\leq d-2$. Then we have the following inclusions.
\begin{align*}
\mathcal{F}_d\ast \mathcal{F}_d &= \Sigma^{-(d-1)}\mathcal{M}\ast \cdots \ast \Sigma^{-1}\mathcal{M}\ast\mathcal{M} \ast \Sigma^{-(d-1)}\mathcal{M}\ast \cdots \ast \Sigma^{-1}\mathcal{M}\ast\mathcal{M}\\
&\subseteq \Sigma^{-(d-1)}\mathcal{M}\ast \cdots \ast \Sigma^{-1}\mathcal{M}\ast \mathcal{F}_d\ast  \Sigma^{-(d-2)}\mathcal{M}\ast\cdots \ast \Sigma^{-1}\mathcal{M}\ast \mathcal{M}\\
&\subseteq \mathcal{F}_d.
\end{align*}
Here, the first inclusion uses the hypothesis, and  the last one follows by applying (\ref{equ:ct1}) repeatedly.
\end{proof}

We emphasize that the condition $Z\in \mathcal{F}_{d-1}$ in Proposition~\ref{prop:filt-ct}(2) below is crucial.

\begin{prop}\label{prop:filt-ct}
   \begin{enumerate}
   \item Any two $\Sigma^{\leq 0}(\mathcal{M})$-filtrations of an object $X$ with length at most $d$  are equivalent.
       \item Let  $X\rightarrow Y \rightarrow Z \rightarrow \Sigma(X)$ be an exact triangle.   Assume that $1\leq n\leq d $, and that
     $$0=X_n\longrightarrow X_{n-1}\longrightarrow \cdots \longrightarrow X_1\longrightarrow X_0=X
$$
and
$$0=Z_n\longrightarrow Z_{n-1}\longrightarrow \cdots \longrightarrow Z_1\longrightarrow Z_0=Z
$$
are $\Sigma^{\leq 0}(\mathcal{M})$-filtrations of $X$ and $Z$, respectively. If $Z$ belongs to $\mathcal{F}_{d-1}$, then $Y$ has a $\Sigma^{\leq 0}(\mathcal{M})$-filtration
$$0=Y_n\longrightarrow Y_{n-1}\longrightarrow \cdots \longrightarrow Y_1\longrightarrow Y_0=Y
$$
with its factors $M_i^Y\simeq M_i^X\oplus M_i^Z$ for $0\leq i\leq n-1$.
       \end{enumerate}
\end{prop}

\begin{proof}
The same proof of Proposition~\ref{prop:filt} yields (1). For (2), since $Z$ belongs to $\mathcal{F}_{d-1}$, by Lemma~\ref{lem:ct1}(2) we do have ${\rm Hom}_\mathcal{T}(Z, \Sigma\mathcal{M})$=0. Then the first square in the proof of Lemma~\ref{lem:HL} trivially commutes. The remaining argument there carries through, and yields the required filtration for $Y$.
\end{proof}

In what follows, we obtain two cluster-tilting analogues of  Proposition~\ref{prop:F}. The following proposition is similar to \cite[Proposition~4.8]{OS}.

\begin{prop}\label{prop:d-ct}
Let $\mathcal{M}$ be a $d$-rigid subcategory of $\mathcal{T}$. Then for each $1\leq m <d$,  the inclusion $\mathcal{M}\hookrightarrow \mathcal{F}_m$ induces an isomorphism $K_0^{\rm sp}(\mathcal{M})\simeq K_0(\mathcal{F}_m)$ of abelian groups.
\end{prop}

\begin{proof}
For each $X\in \mathcal{F}_m$, we choose a $\Sigma^{\leq 0}(\mathcal{M})$-filtration
$$0=X_n\longrightarrow X_{n-1}\longrightarrow \cdots \longrightarrow X_1\longrightarrow X_0=X$$
with $n\leq m$ and factors $M_i^X$.
We define an element
\begin{align}\label{equ:gamma}
\gamma(X)=\sum_{i=0}^{n-1}(-1)^i\langle M_i^X\rangle
\end{align}
in $K_0^{\rm sp}(\mathcal{M})$. By Proposition~\ref{prop:filt-ct}(1), the element $\gamma(X)$ does not depend on the choice of such a filtration; compare \cite[Remark~5.2]{Fed}. This statement holds also for the case $m=d$.

By Proposition~\ref{prop:filt-ct}(2), the map $(X\mapsto \gamma(X))$ is compatible with exact triangles in $\mathcal{F}_m$ for $m<d$. Therefore, such a map induces a well-defined homomorphism $\Gamma\colon K_0(\mathcal{F}_m)\rightarrow K_0^{\rm sp}(\mathcal{M})$ such that $\Gamma([X])=\gamma(X)$. It is routine to verify that $\Gamma$ is inverse to the induced homomorphism $K_0^{\rm sp}(\mathcal{M})\rightarrow K_0(\mathcal{F}_m)$.
\end{proof}

The following remark shows that the condition $m<d$ above is necessary.

\begin{rem}\label{rem:rel}
The induced map $K_0^{\rm sp}(\mathcal{M})\rightarrow  K_0(\mathcal{F}_d)$ is surjective, but not injective in general. We define the \emph{relative Grothendieck group} $K_0^{\rm rel}(\mathcal{F}_d)$ to be the abelian group generated by the set $\{\{C\}\; | \; C\in \mathcal{F}_d\}$ subject to the relations $\{C\}-(\{C_1\}+\{C_2\})$ whenever there is an exact triangle $C_1\rightarrow C\rightarrow C_2\rightarrow \Sigma(C_1)$ in $\mathcal{T}$ with $C_1, C\in \mathcal{F}_d$ and $C_2\in \mathcal{F}_{d-1}$. Then the same argument above yields an isomorphism
$$K_0^{\rm sp}(\mathcal{M})\stackrel{\sim}\longrightarrow K_0^{\rm rel}(\mathcal{F}_d),$$
whose inverse  sends $\{C\}$ to $\gamma(C)$.
\end{rem}

The following immediate consequence of Proposition~\ref{prop:d-ct} somehow complements Proposition~\ref{prop:F}.

\begin{cor}
Let $\mathcal{M}$ be a presilting subcategory of $\mathcal{T}$. Then for any $m\geq 1$,   the inclusion $\mathcal{M}\hookrightarrow \mathcal{F}_m$ induces an isomorphism $K_0^{\rm sp}(\mathcal{M})\simeq K_0(\mathcal{F}_m)$ of abelian groups.
\end{cor}

\begin{proof}
As mentioned before, any presilting subcategory is $d$-rigid for any $d\geq 2$. Then the required result follows from Proposition~\ref{prop:d-ct} immediately.
\end{proof}

Assume that $\mathcal{M}$ is $d$-rigid such that $\mathcal{F}_d$ is closed under extensions. Denote by $N$ the subgroup of $K_0^{\rm sp}(\mathcal{M})$ generated by these elements
$$\gamma(E)-\langle M_1\rangle-(-1)^{d-1}\langle M_2\rangle $$
 for all exact triangles
 \begin{align}\label{tri:E}
 M_1\rightarrow E\rightarrow \Sigma^{-(d-1)}(M_2)\rightarrow \Sigma(M_1)
  \end{align}
  with $M_1, M_2\in \mathcal{M}$. Here, we observe by the assumption above that $E$ belongs to $\mathcal{F}_d$, and refer to (\ref{equ:gamma}) for the definition of $\gamma(E)$. We consider the quotient group $K_0^{\rm sp}(\mathcal{M})/N$, whose typical element is denoted by $\overline{\langle M\rangle}$.

\begin{thm}\label{thm:Fd}
Let $\mathcal{M}$ be a $d$-rigid subcategory of $\mathcal{T}$. Assume that $\mathcal{F}_d$ is closed under extensions. Then the inclusion $\mathcal{M}\hookrightarrow \mathcal{F}_d$ induces an isomorphism $K_0^{\rm sp}(\mathcal{M})/N \simeq K_0(\mathcal{F}_d)$ of abelian groups.
\end{thm}

\begin{proof}
The inclusion $\mathcal{M}\hookrightarrow \mathcal{F}_d$ certainly induces a homomorphism
$$K_0^{\rm sp}(\mathcal{M})/N \longrightarrow K_0(\mathcal{F}_d),$$ which is surjective. To construct its inverse, it suffices to prove the following claim: for each exact triangle $X\stackrel{a} \rightarrow Y\stackrel{b}\rightarrow Z\stackrel{c}\rightarrow \Sigma(X)$ with $X, Y, Z\in \mathcal{F}_d$, we always have
$$\overline{\gamma(Y)}=\overline{\gamma(X)}+\overline{\gamma(Z)}.$$

\emph{Step 1.}\; Assume that $Z$ belongs to $\mathcal{F}_{d-1}$. Proposition~\ref{prop:filt-ct}(2) yields $\gamma(Y)=\gamma(X)+\gamma(Z)$ in $K_0^{\rm sp}(\mathcal{M})$.

\emph{Step 2.}\; Assume that $Z$ belongs to $\Sigma^{-(d-1)}\mathcal{M}$. Fix an exact triangle
$$X_1\stackrel{i}\longrightarrow X\stackrel{p}\longrightarrow M_0^X\longrightarrow \Sigma(X_1),$$
which appears in a $\Sigma^{\leq 0}(\mathcal{M})$-filtration of $X$ with length $d$. In particular, we have $M_0^X\in \mathcal{M}$ and $\Sigma(X_1)\in \mathcal{F}_{d-1}$. By the construction (\ref{equ:gamma}) of $\gamma(X)$, we have
\begin{align}\label{equ:gamma1}
\gamma(X)=\gamma(X_1)+\langle M_0^X\rangle.
\end{align}
By (TR4) and rotations, we have the following commutative diagram.
\[
\xymatrix{
X_1\ar[d]_-{i} \ar@{=}[r] & X_1\ar[d]^-{a\circ i}\\
X\ar[d]_-{p} \ar[r]^-a & Y \ar[d]\ar[r]^-{b} & Z\ar@{=}[d]\ar[r]^-{c} & \Sigma(X)\ar[d]^-{\Sigma(p)}\\
M_0^X\ar[r] \ar[d] &  E\ar[d] \ar[r] & Z\ar[d]^-{c} \ar[r] & \Sigma(M_0^X)\\
\Sigma(X_1)\ar@{=}[r] & \Sigma(X_1)\ar[r]^{\Sigma(i)} & \Sigma(X)
}\]
Here, the third row and the second column from the left are both exact triangles. Since $\mathcal{F}_d$ is closed under extensions,  the third row implies that $E$ belongs to $\mathcal{F}_d$.  Recall that $Z$ belongs to $\Sigma^{-(d-1)}\mathcal{M}$. The very definition of the subgroup $N$ yields
\begin{align}\label{equ:gamma2}
\overline{\gamma(E)}=\overline{\langle M_0^X\rangle}+\overline{\gamma(Z)}.
\end{align}
By rotating the second column, we have an exact triangle $Y\rightarrow E\rightarrow \Sigma(X_1)\rightarrow \Sigma(Y)$. Since $\Sigma(X_1)\in \mathcal{F}_{d-1}$, Step 1 yields
\begin{align}\label{equ:gamma3}
\gamma(E)=\gamma(Y)+\gamma(\Sigma(X_1))=\gamma(Y)-\gamma(X_1).
\end{align}
By combining (\ref{equ:gamma1}), (\ref{equ:gamma2}) and (\ref{equ:gamma3}), we obtain the required equality.

\emph{Step 3.}\; We now treat  the general case. Using the $\Sigma^{\leq 0}(\mathcal{M})$-filtration of $Z$ with length $d$, we obtain an exact triangle
$$\Sigma^{-(d-1)}(M_{d-1}^Z)\stackrel{j}\longrightarrow Z\stackrel{q}\longrightarrow Z'\longrightarrow \Sigma^{2-d}(M_{d-1}^Z)$$
with  $M_{d-1}^Z\in \mathcal{M}$ and $Z'\in \mathcal{F}_{d-1}$. Moreover, by the construction (\ref{equ:gamma}) of $\gamma(Z)$, we have
\begin{align}\label{equ:gamma4}
\gamma(Z)=\gamma(Z')+(-1)^{(d-1)}\langle M_{d-1}^Z\rangle.
\end{align}
By (TR4) and rotations, we have the following commutative diagram.
\[\xymatrix{
X\ar@{=}[d]\ar[r]& F\ar[d]\ar[r] & \Sigma^{-(d-1)}(M_{d-1}^Z)\ar[d]^-{j} \ar[r] & \Sigma(X)\ar@{=}[d]\\
X\ar[r]^-a & Y\ar[d] \ar[r]^-b & Z\ar[d]^-{q} \ar[r]^-c & \Sigma(X)\ar[d]\\
& Z'\ar[d]\ar@{=}[r] & Z'\ar[d]\ar[r]& \Sigma(F)\\
& \Sigma(F)  \ar[r]& \Sigma^{2-d}(M_{d-1}^Z)
}\]
The second column from the left is an exact triangle. Since $Z'$ belongs to $\mathcal{F}_{d-1}$,  Step 1 yields
\begin{align}\label{equ:gamma5}
\gamma(Y)=\gamma(F)+\gamma(Z').
\end{align}
Applying Step 2 to the first row, we have
\begin{align}\label{equ:gamma6}
\overline{\gamma(F)}=\overline{\gamma(X)}+ (-1)^{d-1} \overline{\langle M_{d-1}^Z\rangle}.
\end{align}
Combining (\ref{equ:gamma4}), (\ref{equ:gamma5}) and (\ref{equ:gamma6}), we obtain the required equality. This proves the claim, and completes the proof.
\end{proof}

\begin{rem}
(1) We mention that Theorem~\ref{thm:Fd} actually  implies Proposition~\ref{prop:d-ct}. To be more precise,  we assume that $\mathcal{M}$ is $d'$-rigid with $d'>d$. Since we have ${\rm Hom}_\mathcal{T}(\Sigma^{-(d-1)}\mathcal{M}, \Sigma\mathcal{M})=0$, by Lemma~\ref{lem:fd-ext} $\mathcal{F}_d$ is closed under extensions; moreover,  the corresponding subgroup $N$ of $K_0^{\rm sp}(\mathcal{M})$ is zero. Then the isomorphism in  Theorem~\ref{thm:Fd} yields  the required isomorphism  $K_0^{\rm sp}(\mathcal{M})\simeq K_0(\mathcal{F}_d)$ in Proposition~\ref{prop:d-ct}.

(2) Assume that $\mathcal{F}_d$ is closed under extensions. Combining the isomorphisms in Remark~\ref{rem:rel} and Theorem~\ref{thm:Fd}, we obtain the following one
$$K_0^{\rm rel}(\mathcal{F}_d)/N'\stackrel{\sim}\longrightarrow K_0(\mathcal{F}_d),\;  \{C\}+N'\mapsto [C].$$
Here, $N'$ is the subgroup of $K_0^{\rm rel}(\mathcal{F}_d)$ generated by the elements $$\{E\}-\{M_1\}-(-1)^{d-1}\{M_2\}$$
for all exact triangles $M_1\rightarrow E\rightarrow \Sigma^{-(d-1)}(M_2)\rightarrow \Sigma(M_1)$ with $M_1, M_2\in \mathcal{M}$.
\end{rem}

Following \cite[Section~3]{IYo} and  \cite[Definition~5.1]{Kvamme}, a $d$-rigid subcategory $\mathcal{M}$ of $\mathcal{T}$ is called {\em $d$-cluster-tilting} provided that $\mathcal{F}_d=\mathcal{T}$. The condition is equivalent to $\mathcal{T}=\Sigma^{-(d-1)}\mathcal{M}\ast\cdots \Sigma^{-1}\mathcal{M}\ast \mathcal{M}$, which is further equivalent to $\mathcal{T}=\mathcal{M}\ast \Sigma\mathcal{M}\ast \cdots \ast \Sigma^{d-1}\mathcal{M}$ by rotations.

By the proof of \cite[Proposition~5.3]{Kvamme}, we observe the following fact: a subcategory $\mathcal{M}$ is  $d$-cluster-tilting and closed under direct summands if and only if $\mathcal{M}$ is contravariantly finite in $\mathcal{T}$ and $\mathcal{M}=\{X\in \mathcal{T}\ | \ {\rm Hom}_{\mathcal{T}}(\mathcal{M}, \Sigma^i(X))=0, 1\leq i<d\}$, if and only if $\mathcal{M}$ is covariantly finite in $\mathcal{T}$ and $\mathcal{M}=\{Y\in \mathcal{T}\ | \ {\rm Hom}_{\mathcal{T}}(Y, \Sigma^i\mathcal{M})=0, 1\leq i<d\}$; compare \cite[Proposition~2.2.2]{Iyama}.

We mention that $2$-cluster-tilting objects play an important role in various additive categorifications \cite{BMRRT, GLS} of cluster algebras. For $d$-cluster-tilting objects in higher cluster categories, we refer to \cite{Wr, ZZ}.

We have  the following immediate consequence of  Theorem~\ref{thm:Fd}, which is a cluster-tilting analogue of Theorem~A, and is similar to \cite[Theorem~C]{Fed} and \cite[Theorem~5.22]{OS}.

\begin{cor}\label{cor:cluster-tilting}
 Let $\mathcal{M}$ be a  $d$-cluster-tilting subcategory of $\mathcal{T}$. Then the inclusion $\mathcal{M}\hookrightarrow \mathcal{T}$ induces an isomorphism $K_0^{\rm sp}(\mathcal{M})/N \simeq K_0(\mathcal{T})$ of abelian groups. \hfill $\square$
	\end{cor}

In the following remark, we mention that Corollary~\ref{cor:cluster-tilting} recovers \cite[Theorem~C]{Fed}.

\begin{rem}
 Assume that $\mathcal{M}$ is a $d$-cluster-tilting subcategory of $\mathcal{T}$ satisfying $\Sigma^d(\mathcal{M})=\mathcal{M}$. Then it is naturally a $(d+2)$-angulated category in the sense of \cite{GKO}. We claim that any triangle of the form (\ref{tri:E}) and a $\Sigma^{\leq 0}(\mathcal{M})$-filtration of $E$ with length $d$ induce a $(d+2)$-angle; moreover, any $(d+2)$-angle arises in this way.

 We take $d=3$ for example. Assume that $M_1\stackrel{a} \rightarrow E \stackrel{b}\rightarrow \Sigma^{-2}(M_2)\stackrel{c}\rightarrow \Sigma(M_1)$ is an exact triangle with $M_1, M_2\in \mathcal{M}$. The following two exact triangles
  $$
X_1\stackrel{i_1} \rightarrow E \stackrel{p_0}\rightarrow M_0^E\stackrel{q_0}\rightarrow \Sigma(X_1) \mbox{ and } \Sigma^{-2}(M_2^E)\stackrel{i_2} \rightarrow X_1 \stackrel{p_1}\rightarrow \Sigma^{-1}(M_1^E) \stackrel{q_1} \rightarrow \Sigma^{-1}(M_2^E)
 $$
 define a  $\Sigma^{\leq 0}(\mathcal{M})$-filtration of $E$ with length $3$. Then by \cite[Theorem~1]{GKO}, we have the following induced $5$-angle in $\mathcal{M}$.
 $$\Sigma^{-3}(M_2)\xrightarrow{\Sigma^{-1}(c)}M_1 \xrightarrow{p_0\circ a}  M_0^E \xrightarrow{\Sigma(p_1)\circ q_0}  M_1^E \xrightarrow{\Sigma(q_1)} M_2^E \xrightarrow{\Sigma^2(b\circ i_1\circ i_2)}  M_2$$
Here, by the assumption  above we have that $\Sigma^{-3}(M_2)$ belongs to $\mathcal{M}$. By reversing the argument, we infer that any $5$-angle in $\mathcal{M}$ arises in this way.

  By combining the claim above and \cite[Proposition~5.4]{Fed}, we infer that the above subgroup $N$ coincides with the group ${\rm Im}\; \theta_\mathcal{M}$ defined in \cite{Fed}. Then the isomorphism in  Corollary~\ref{cor:cluster-tilting} yields the one in \cite[Theorem~C]{Fed}.
\end{rem}

The following trivial example indicates that the extension-closed condition in Theorem~\ref{thm:Fd} is somehow weaker than the one in Corollary~\ref{cor:cluster-tilting}.

\begin{exm}
{\rm Let $d\geq 2$. Let $\mathcal{T}'$ be a triangulated category with a $d$-cluster tilting subcategory $\mathcal{M}'$. Let $\mathcal{T}''$ be another triangulated category  and $\mathcal{M}''\subseteq \mathcal{T}''$ be a presilting subcategory or a $d'$-rigid subcategory with $d<d'$. Denote by  $\mathcal{F}''_d$ the subcategory  formed by objects admitting a $\Sigma^{\leq 0}(\mathcal{M}'')$-filtrations of length at most $d$; it is closed under extensions in $\mathcal{T}''$.

Set $\mathcal{T}=\mathcal{T}'\times \mathcal{T}''$ to be the product category. Then $\mathcal{M}=\mathcal{M}'\times \mathcal{M}''$ is a $d$-rigid subcategory of $\mathcal{T}$, which is not necessarily $d$-cluster-tilting. Recall that $\mathcal{F}_d$ denotes the full subcategory of $\mathcal{T}$ formed by those objects, which  admit a $\Sigma^{\leq 0}(\mathcal{M})$-filtration of length at most $d$. We have $\mathcal{F}_d=\mathcal{T}'\times \mathcal{F}''_d$, which is closed under extensions in $\mathcal{T}$.}
\end{exm}

\vskip 10pt

\noindent{\bf Acknowledgements}. \quad We are grateful to the anonymous referee for many helpful comments. X.W. thanks Mathematisches Forschungsinstitut Oberwolfach for the excellent working condition, where this work is partly done.  We thank Professor Hongxing Chen for the references \cite{Sch, So}, and thank Professor Yu Zhou and Professor Bin Zhu for \cite{OS} and helpful discussion.  This project is supported by  the National Natural Science Foundation of China (No.s 12325101, 12171207, 12131015, 12371015 and 12161141001).

\bibliography{}

\vskip 10pt

 {\footnotesize \noindent  Xiao-Wu Chen\\
 School of Mathematical Sciences, University of Science and Technology of China, Hefei 230026, Anhui, PR China\\

 \footnotesize \noindent Zhi-Wei Li, Xiaojin Zhang\\
School of Mathematics and Statistics, Jiangsu Normal University, Xuzhou 221116, Jiangsu, PR China\\

 \footnotesize \noindent Zhibing Zhao\\
 School of Mathematical Sciences, Anhui University,  Hefei 230601, Anhui, PR China}

\end{document}